\documentclass[11pt,a4paper]{article}
\usepackage{amsfonts,amssymb,amsmath,amscd,stmaryrd,latexsym,color,makeidx,theorem,psfrag}
\usepackage{hyperref}

\author{Azahara DelaTorre \\ {\small  Albert-Ludwigs-Universit\"at Freiburg} \and Ali Hyder \\ {\small UBC Vancouver}\and  Luca Martinazzi \\  {\small Universit\'a di Padova} \and Yannick Sire \\ \small{John Hopkins University} }

\title{The non-local mean-field equation on an interval}

\newtheorem{trm}{Theorem}
\newtheorem{prop}[trm]{Proposition}
\newtheorem{cor}[trm]{Corollary}
\newtheorem{lemma}[trm]{Lemma}

\newcommand{\vp}{\varphi}

\newcommand{\R}{\mathbb{R}}

\newcommand{\de}{\partial}
\newcommand{\ve}{\varepsilon}

\newcommand{\M}[1]{\mathcal{#1}}

\newcommand{\bra}[1]{\left({#1}\right)}
\newcommand{\D}{\Delta}
\newenvironment{proof}{\noindent\emph{Proof.}}{\hfill$\square$\medskip}

\DeclareMathOperator{\loc}{loc}

\begin{document}
\maketitle

\begin{abstract}
We consider the fractional mean-field equation on the interval $I=(-1,1)$
$$(-\Delta)^\frac{1}{2} u=\rho\frac{e^{u}}{\int_{I}e^{u}dx},$$
subject to Dirichlet boundary conditions,
and prove that existence holds if and only if $\rho <2\pi$. This requires the study of blowing-up sequences of solutions. We provide a series of tools in particular which can be used (and extended) to higher-order mean field equations of non-local type.
\end{abstract}

\section{Introduction}
Given a number $\rho>0$, we consider the non-local mean-field equation
\begin{equation}\label{eq0}
(-\Delta)^\frac{1}{2} u=\rho\frac{e^{u}}{\int_{I}e^{u}dx},\quad I=(-1,1)
\end{equation}
subject to the Dirichlet boundary condition
\begin{equation}\label{dir}
u\equiv 0\quad \textrm{in } \R\setminus I. %,\quad u\in H^{\frac{1}{2},2}(\R).
\end{equation}
There are different ways to define the fractional Laplacian $(-\Delta)^\frac12$ and therefore make sense of Problem \eqref{eq0}-\eqref{dir}.
Consider the space of functions $L_{\frac12}(\mathbb{R})$ defined by
\begin{equation}
L_{\frac12}(\mathbb{R})=\left\{ u\in L^{1}_{\loc}(\mathbb{R}): \int_{\mathbb{R}}\frac{|u(x)|}{1+|x|^2}dx <\infty \right\}.
\end{equation}
For a function $u\in L_{\frac12}(\mathbb{R})$  one can define $(-\Delta)^{\frac12}u$ as a tempered distribution as follows:
\begin{equation}\label{deffrlap}
\langle (-\Delta)^{\frac12}u,\varphi \rangle := \int_{\R{}} u(-\Delta)^{\frac12}\varphi dx, \quad \varphi\in \M{S},
\end{equation}
where $\M{S}$ denotes the Schwartz space of rapidly decreasing smooth functions and for $\varphi\in\M{S}$ we set
$$(-\Delta)^\frac12\varphi:=\mathcal{F}^{-1}(|\cdot |\hat{\varphi}).$$ Here the Fourier transform is defined by
$$ \hat \varphi(\xi)\equiv \mathcal{F} \varphi(\xi):=\frac{1}{\sqrt{2\pi}}\int_{\R{}}e^{-ix\xi}\varphi(x)\, dx.$$
Notice that the convergence of the integral in \eqref{deffrlap} follows from the fact that for $\varphi\in \M{S}$ one has
$$|(-\Delta)^\frac12\varphi(x)| \le C(1+|x|^{2})^{-1}.$$
If $u\in C^{0,\alpha}(I)$ we can also define
$$(-\Delta)^\frac12 u(x):=\frac{1}{\pi}P.V. \int_{\R} \frac{u(x)-u(y)}{(x-y)^2}dy,\quad x\in I.$$

%A third way to define $(-\Delta)^\frac12 u$ is to define $U\in \dot H^1(\R^2_+)$ via the Poisson kernel, so that it satisfies
%\begin{align*}
%&\Delta U =0 \quad \R\times (0,\infty),\\
%&U(x,0)=u(x)
%\end{align*}
%and set
%\begin{equation}\label{frlap}
%(-\Delta)^\frac{1}{2} u := - \frac{\de U}{\de y}.
%\end{equation}
These definitions are equivalent for the functions that we shall consider, namely function in $C^{0,\frac12}(\R)$ vanishing outside $I$. In fact, every solution to \eqref{eq0}-\eqref{dir} lies in $C^{0,\frac12}(\R)$, see e.g. Corollary 1.6 of \cite{RS}, and it is smooth inside $I$ by a standard bootstrap argument. Therefore there is no loss of generality in working only with functions in $C^{0,\frac12}(\R)\cap C^\infty(I)$.

\medskip

In this paper we shall develop some tools to treat existence and non-existence for problem \eqref{eq0}-\eqref{dir}. In spite of the possibility of working with the extention of $u$ to the upper half-plane, i.e. of localizing the problem as often done, we will only use \emph{purely non-local} methods, that can be best extended to treat also non-local \emph{higher-dimensional} cases. 

\medskip

In dimension $2$ the analog of Problem \eqref{eq0}-\eqref{dir} is
\begin{equation}\label{eq2d}
-\Delta u=\rho\frac{e^{u}}{\int_{\Omega}e^{u}dx}\text{ in }\Omega,\quad u=0\text{ on }\de\Omega,\quad \Omega\Subset\R^2
\end{equation}
where $\Omega$ is smoothly bounded. As proven in \cite{CLMP} using variational arguments (minimization of a suitable functional) and in \cite{Kie} via probabilistic methods, Problem \eqref{eq2d} has a solution for every $\rho\in (0,8\pi)$.
The threshold $8\pi$ is sharp since when $\Omega$ is star-shaped \eqref{eq2d} has no solution for every $\rho\ge 8\pi$ by the Pohozaev identity.

If, on the other hand, $\Omega$ is not simply connected or it is replaced by a closed Riemann surface $(\Sigma,g)$ of genus at least $1$, in which case \eqref{eq2d} is replaced by
\begin{equation}\label{eq2d'}
-\Delta_g u=\rho\left(\frac{e^{u}}{\int_{\Sigma }e^{u}dv_g}-1\right)\text{ in }\Sigma,
\end{equation}
Ding-Jost-Li-Wang \cite{DJLW} proved that \eqref{eq2d'} admits a solution for every $\rho\in (8\pi,16\pi)$. Struwe and Tarantello \cite{ST} independently proved a similar result on the flat torus and for $\rho\in (8\pi, 4\pi^2)$. For a general closed surface (including a sphere) Malchiodi \cite{Mal} proved existence for every $\rho\not\in 8\pi \mathbb{N}$, using the barycenter technique, see also \cite{Dja}.

\medskip

An important tool in proving such existence results is an a priori study of the blowing-up behavior of sequences $(u_k)$ of solutions to \eqref{eq2d} or \eqref{eq2d'} with $\rho=\rho_k$. This was performed by Brezis-Merle \cite{BM} and Li-Shafrir \cite{LS} for the Liouville equation, which arises from \eqref{eq2d} by adding a constant. Theses seminal works have several extensions to even dimension $4$ and higher, see e.g. \cite{wei},\cite{RW} and \cite{MP}, using higher-dimensional compactness results, see e.g. \cite{mar2}. In order to study the $1$-dimensional case we will need the following analogue non-local blow-up result.

\begin{trm}\label{trm1} Let $u_k$ be a sequence of solutions to \eqref{eq0}, \eqref{dir} with $\rho=\rho_k>0$.
Then up to a subsequence one of the following is true:
\begin{itemize}
\item[(i)]  $(u_k)$ is bounded in $C^{0,\tfrac12}(\R)\cap C^\ell_{\loc}(I)$ for every $\ell\in\mathbb{N}$. %, {\color{blue}and $\lim_{k\to\infty}\rho_k<2\pi$}.
\item[(ii)] $\lim_{k\to\infty} u_k(0)=\infty$
\begin{equation}\label{quant_1}
\rho_k\uparrow2\pi\quad \text{as }k\to \infty.
\end{equation}
Moreover,  for $0<\sigma<\frac12$
\begin{equation}\label{green_1}
u_k\to 2\pi G_{0}\quad \text{in }C^{0,\sigma}_{\loc}(\R \backslash \{0\}),
\end{equation}
where $G_0$ is the Green function of $(-\Delta)^\frac12$ on $I$ with Dirichlet boundary condition.
\end{itemize}
\end{trm}

Let us notice that if we replace the right-hand side of \eqref{eq0} with the nonlinearity $e^{u^2}$, nonlocal compactness problems have been studied in \cite{IMM} and \cite{MMS}, but the techniques used there are different, for instance because of the lack of a Pohozaev-type identity. In fact a result analog to \eqref{green_1} is still unknown in the fractional case, although in dimension $2$ it was recently proven by Druet-Thizy \cite{DT}, see also \cite{MM}.

\medskip
%It is interesting to notice that if blow-up occurs, then $\rho_k\to 2\pi$ from below.

Using Theorem \ref{trm1} and Schauder's fixed-point theorem we are able to prove the following result about existence and non-existence.

\begin{trm}\label{trm2}
There exists a non-trivial non-negative  solution $u=u_\rho$ to \eqref{eq0}\eqref{dir} if and only if $\rho\in (0,2\pi)$. Moreover, $$u_\rho(0)\to\infty\quad\text{as }\rho\uparrow 2\pi.$$
\end{trm}

Although our method is topological, it is plausible that a variational argument in the spirit of \cite{CLMP} can also be employed.

The non-existence for $\rho\ge 2\pi$ follows at once from a Pohozaev-type inequality (see Proposition \ref{poho}), consistently with the non-existence in dimension $2$ when the domain is star-shaper. Notice that the critical threshold $2\pi$ in Theorem \ref{trm2} corresponds to the value $8\pi$ for Problems \eqref{eq2d} and \eqref{eq2d'}.

The last statement of Theorem \ref{trm2} is about the existence of blowing-up sequences of solutions, namely it shows that the situation presented in Case \emph{(ii)} of Theorem \ref{trm1} actually occurs. The proof will follow by contradiction, together with the non-existence result of $\rho=2\pi$. In dimension $2$ and higher, several such results (sometimes very subtle) are obtained by the Lyapunuv-Schmidt reduction. For instance, when $\Omega$ is simply connected, Weston \cite{Wes} proved existence of solutions to \eqref{eq2d} blowing-up on a critical point of the Robin function of $\Omega$, and \cite{DJLW} extended this result to the non-simply connected case. Multi-peak solutions have also been constructed, starting with the seminal work of Baraket-Pacard\cite{BP}, see e.g. the work \cite{Esp} and its references.

We also mention that in dimension $2$, when $\Omega$ is simply connected and $\rho\in (0,8\pi)$, Suzuki \cite{Suz} proved uniqueness of solutions for Problems \eqref{eq2d}. It is reasonable to expect that the same holds in $1$ dimension for \eqref{eq0}-\eqref{dir}.

%-- In \cite{CLMP2} it is proven that if $\Omega$ is close to a disk there are solutions that blow up with $\rho_k\to 8\pi^-$.

\paragraph{Acknowledgements} We are grateful to Francesca Da Lio for reading the manuscript and for very useful remarks and to Gabriele Mancini for interesting conversations.

The first, second and third author have been supported by the Swiss National Science Foundation projects n. PP00P2-144669, PP00P2-170588/1 and P2BSP2-172064. The first author is also partially supported by Spanish government grants MTM2014-52402-C3-1-P and MTM2017-85757-P. The fourth author is supported by a Simons fellowship.

\section{Preliminaries}

We shall use the Green function defined by the formula
\begin{align}\label{Green}
G_x(y)&:=\tfrac{1}{\pi}\left(\log(\sqrt{(1-|x|^2)(1-|y|^2)}+1-xy)-\log|x-y|\right)  \notag\\
&=:-\tfrac{1}{\pi}\log|x-y|+H(x,y), \quad x,y\in I
\end{align}
and $G_x(y)=0$ for $x\in I$, $y\in \R\setminus I$.
It is well-known (see e.g. \cite{Green1}) that
\begin{equation}\label{DeltaGreen}
(-\Delta)^\frac12 G_x = \delta_x\quad \text{for }x\in I.
\end{equation}

As usual, using the Green function we can write solutions to \eqref{eq0}-\eqref{dir} in terms of a Green representation formula.

\begin{lemma}\label{green} A function $u\in C^\frac{1}{2}(\R)\cap C^\infty(I)$ solves \eqref{eq0}-\eqref{dir} if and only if
$$u(x)=\rho \int_I G_x(y) \frac{e^{u(y)}}{\int_{I}e^{ u(\xi)}d\xi} dy.$$
\end{lemma}

\begin{proof} This standard proof can be found for instance in the proof of  \cite[Proposition 7]{MarAMT} (Identity (15) in particular).
\end{proof}

\begin{cor}\label{upos} If $u$ solves \eqref{eq0}-\eqref{dir}, then $u> 0$ in $I$.
\end{cor}

In the following lemma we apply a non-local version of the famous moving-plane technique.

\begin{lemma}\label{MP} Let $u\in C^\frac{1}{2}(\R)\cap C^\infty(I)$ solve \eqref{eq0}-\eqref{dir}. Then $u$ is even and decreasing, in the sense that $u(x)=u(-x)$ and $u(x)\ge u(y)$ for $0\le x\le y$.
\end{lemma}

\begin{proof}
This follows at once from the moving plane technique, see Theorem \ref{MP_Ap} in the Appendix.
\end{proof}

\begin{prop}\label{poho}
Let $\hat{u}\in C^\frac{1}{2}(\R)\cap C^\infty(I)$ be a solution to $$\hat u(x)=\int_I G_x(y)e^{\hat u(y)}dy+c,$$ for some $c\in\R$. Then for $$ \rho:=\int_I e^{\hat u(y)}dy,$$ we have $\rho<2\pi$.
\end{prop}
\begin{proof}
 We fix $\psi\in C^1((0,\infty))$ such that $\psi=0$ on $[0,1)$ and $\psi=1$ on $(2,\infty)$. Set for $\ve >0$ small enough, $\psi_\ve(x):=\psi(\frac x\ve)$. 
% Since {\color{magenta} $\hat{u}\in$ use the space given in the introduction. Make reference to  $\rho<\infty$}.
We can rewrite $\hat u$ as
 \begin{align}\label{hatu}\hat u(x)=\frac1\pi\int_I\log\left(\frac{1}{|x-y|}\right)\psi_\ve(|x-y|)e^{\hat u(y)}dy+\int_{I}H(x,y)e^{\hat u(y)}dy+w(x)+c,\end{align}
where 
\begin{equation}\label{w}
w(x):=w_{\psi,\ve}(x):=\frac1\pi\int_I\log\left(\frac{1}{|x-y|}\right)(1-\psi_\ve(|x-y|))e^{\hat u(y)}dy.
\end{equation}
Note that by definition of $\psi_\ve$ we integrate only on $[-2\ve,2\ve]$, so we obtain
$$\|w\|_{L^\infty(I)}\leq C\ve|\log\ve| \|e^{\hat u}\|_{L^\infty(I)}.$$
Moreover, $w\in C^1(I)$, which follows from  $\hat u \in C^1(I)$. %{\color{magenta}  Indeed, by Morrey's embedding $\hat u\in C^{0,1/2}$ and then by standard bootstrap argument and Schauder estimates we obtain the claimed regularity (see \cite{CS1}) }.
Differentiating under the integral sign in \eqref{hatu} we get
 $$\hat u'(x)=\frac1\pi \int_I\frac{\partial }{\partial  x}\left( \log\frac{1}{|x-y|}\psi_\ve(|x-y|)\right)e^{\hat u(y)}dy+\int_I\frac{\partial }{\partial  x}H(x,y)e^{\hat u(y)}dy+w'(x).$$ 
 We define $I_1$ as the quantity that we obtain multiplying the above identity by $xe^{\hat u(x)}$ and integrating over $I$, i.e,
 \begin{equation*}
 I_1:=\int_I x\hat u'(x)e^{\hat u(x)}dx.
 \end{equation*}
 On the one hand, since $\hat{u}$ is even by Lemma \ref{MP}, integration by parts yields 
 \begin{equation}\label{I_1_1}
 I_1=2e^{\hat u(1)}-\int_I e^{\hat u}dx=2e^{\hat u(1)}-\rho.
 \end{equation}
 On the other hand, by definition
 \begin{align*}
 I_1&=\frac1\pi\int_I \int_Ix\frac{\partial }{\partial  x}\left( \log\frac{1}{|x-y|}\psi_\ve(|x-y|)\right)e^{\hat u(y)}e^{\hat u(x)}dydx \\ 
 &\qquad+\int_I\int_Ix\frac{\partial }{\partial  x}H(x,y)e^{\hat u(y)}e^{\hat u(x)}dydx+\int_Iw'(x)xe^{\hat u(x)}dx\\
 &=:I_2+I_3+I_4.
 \end{align*}Using that  $\psi_\ve=0$ on $[0,\ve]$ we obtain 
 \begin{align*}
 I_2&=\frac1\pi\int_I\int_I x\left(- \frac{\psi_\ve(|x-y|)}{x-y}+\log\frac{1}{|x-y|}\psi_\ve'(|x-y|)\frac{x-y}{|x-y|}\right)e^{\hat u(y)}e^{\hat u(x)}dydx\\
 &=-\frac{1}{2\pi}\int_I\int_ I \psi_\ve(|x-y|)e^{\hat u(y)}e^{\hat u(x)}dydx-\frac{1}{2\pi}\int_I\int_ I F(x,y)dydx\\&\quad +\frac{1}{2\pi}\int_I\int_ I \log\frac{1}{|x-y|}|x-y|\psi_\ve'(|x-y|) e^{\hat u(y)}e^{\hat u(x)}dydx \\
 &=:J_1+J_2+J_3,
 \end{align*} 
 where 
 $$F(x,y):=\frac{x+y}{x-y}\left( \psi_\ve(|x-y|)-\log\frac{1}{|x-y|}\psi_\ve'(|x-y|)|x-y|\right)e^{\hat u(y)}e^{\hat u(x)}.$$ 
By dominated convergence theorem, using the definition and regularity of $\psi$ we can assert $$J_1\xrightarrow{\ve\to0} -\frac{\rho^2}{2\pi},\quad J_3\xrightarrow{\ve\to0}0.$$ Moreover, since $F(x,y)=-F(y,x)$, we have   $J_2=0$. Therefore, we get 
\begin{equation}\label{I_2}
I_2\xrightarrow{\ve\to0}-\frac{\rho^2}{2\pi}.
\end{equation}
We claim now that $I_3<0$. To prove it, we first compute
\begin{align*}
x\frac{\partial }{\partial x}H(x,y)&=x\frac1\pi\frac{\partial }{\partial x}\log\left( \sqrt{(1-x^2)(1-y^2)}+1-xy\right)\\
&=\frac{x}{\pi}\frac{-y-x\sqrt{\frac{1-y^2}{1-x^2}}}{\sqrt{(1-x^2)(1-y^2)}+1-xy}\\
&\leq \frac{x}{\pi}\frac{-y }{\sqrt{(1-x^2)(1-y^2)}+1-xy},
%&=:\frac x\pi R_1(x,y).
\end{align*} 
This inequality together with Lemma \ref{MP} (which implies $\hat u(-x)=\hat u(x)$) prove the claim as follows
\begin{align*}
I_3&\leq \frac{-1}{\pi}\int_I\int_I\frac{xy }{\sqrt{(1-x^2)(1-y^2)}+1-xy}e^{\hat u_k(y)}e^{\hat u_k(x)}dydx\\
&=:\frac{-2}{\pi}\int_0^1\int_0^1 K(x,y)e^{\hat u_k(y)}e^{\hat u_k(x)}dydx\\
&<0,
\end{align*} where the last inequality follows from $$K(x,y):=xy\left( \frac{1 }{\sqrt{(1-x^2)(1-y^2)}+1-xy}-\frac{1 }{\sqrt{(1-x^2)(1-y^2)}+1+xy}\right)>0$$  on $ (0,1)\times(0,1)$. 

Finally we show that $I_4\to0$ as $\ve\to 0$. Indeed, integration by parts and the bound for the function $w$ defined in \eqref{w} hold 
 \begin{align*}
 I_4=-\int_I (1+\hat u')e^{\hat u}w dx+o_\ve(1)=o_\ve(1)+o_\ve(1)\int_I |\hat u'|e^{\hat u}dx \xrightarrow{\ve\to0}0,
 \end{align*} 
where we used that $\hat u'e^{\hat u}\in L^1(I)$.
% This should follow from regularity theory, right? Maybe we should have $|\hat u'(x)|\leq C(1-|x|)^{-\frac12}?$ 
Indeed, by Lemma \ref{MP},   $\hat u'\leq 0$ on $(0,1)$,  so we have  $$\int_0^1 |\hat u'|e^{\hat u}dx=\int_0^1 -(e^{\hat u})'dx=(e^{\hat u(0)}-1)<\infty.$$
%{\color{magenta}where $c>0$ is a constant to make the exponent positive. We need to refer here the regularity given in the introduction to justify the inequality.}
Summarising, we obtain that 
$$I_1=I_2+I_3+I_4<I_2+I_4\xrightarrow{\ve\to0}-\frac{\rho^2}{2\pi},$$
The proposition follows immediately from \eqref{I_1_1}. 
\end{proof}

\section{Proof of Theorem \ref{trm1}}
 We set 
\begin{align}\label{ukhat}
\hat u_k:=u_k-\alpha_k,\quad \alpha_k:=\log\bigg(\frac{\int_I e^{u_k}dx}{\rho_k}\bigg).
\end{align}
Using Lemma \ref{green} we write
\begin{equation}\label{ukhatrep}
\hat u_k(x)=\int_I G_x(y)e^{ \hat u_k(y)}dy-\alpha_k,\quad \int_Ie^{\hat u_k}dx=\rho_k,
\end{equation}
and
\begin{equation}\label{ukrep}
u_k(x)=\int_I G_x(y)e^{ \hat u_k(y)}dy.
\end{equation}

If $\hat u_k(0)\leq C$, then by \eqref{ukrep} $u_k$ is bounded in $C^{0,\alpha}([-1,1])$ for every $\alpha\in [0,\tfrac{1}{2}]$ and in $C^\ell_{\loc}(-1,1)$ for $\ell\ge 0$, so that possibility $(i)$ in the theorem occurs. 

%{\color{blue} Two cases of Theorem 1: From \eqref{ukrep} it is clear that  $u_k(0)\to\infty\Rightarrow \hat u_k(0)\to\infty$. However, the implication   $u_k(0)\leq C\Rightarrow \hat u_k(0)\leq C$ is not that clear. To prove this we need to use $iv)$ of Lemma \ref{conveta}. In this way we get $u_k(0)\to\infty\Leftrightarrow \hat u_k(0)\to\infty$

%So, in Lemma \ref{conveta} we assume  $\hat u_k(0)\to\infty$?}

In the following we assume that $\hat u_k(0)\to\infty$ and we shall set
$$r_k:=2e^{-\hat u_k(0)}\to0.$$

\begin{lemma}\label{conveta} Assume that $\hat u_k(0)\to\infty$. Then we have 
\begin{itemize}
%\item[i)]  $r_k:=2e^{-\hat u_k(0)}\to0$
\item[i)] $r_k u_k(0)\to 0$.
\item[ii)] $\eta_k(x):=\hat u_k(r_k x)+\log(r_k) \to\eta_0(x):=\log\frac{2}{1+x^2}\quad \text{in }C^{\infty}_{loc}(\R).$
\item[iii)] $\lim_{R\to\infty}\lim_{k\to\infty}\int_{-Rr_k}^{Rr_k}e^{\hat u_k}dx=2\pi.$
\item[iv)] $\alpha_k\to\infty$. 
\item[v)] $\hat u_k\to-\infty$ in $C^0_{loc}(\bar I\setminus\{0\})$.
\end{itemize}
\end{lemma}
\begin{proof}

\noindent\textbf{Step 1} We show that $r_k u_k(0)\to0$.

Indeed from \eqref{ukrep}, for every $\delta>0$ \begin{align*}u_k(0)&= \left(\int_{|y|<\delta}+\int_{\delta<|y|<1}\right)G_0(y)e^{\hat u_k(y)}dy\\
 &\leq  e^{\hat u_k(0)}\int_{|y|<\delta}G_0(y)dy+e^{\hat u_k(\delta)}\int_{\delta<|y|<1}G_0(y)dy\\ &\leq C e^{\hat u_k(0)}\delta|\log\delta|+e^{\hat u_k(\delta)} .\end{align*} 
Note that for both inequalities  we have used that, by Lemma \ref{MP}, $\hat{u}$ is decreasing on $|y|$.
 Since $\delta>0$ is arbitrary small, and $\hat u_k(\delta)\not \to \infty$ (otherwise Proposition \ref{poho} would be violated), multiplying both sides of the inequality by $r_k$, letting $k\to \infty$ and $\delta\to 0$ we complete the proof of $i)$.
 
\medskip 

We will divide the proof of part $ii)$ in three main steps:

\noindent\textbf{Step 2} For every $\ve>0$ there exists $R\gg 1$ such that for $k$ large $$\int_{|x|>R}\frac{|\eta_k(x)|}{1+x^2}dx<\ve.$$

On the one hand, by definition of $\eta_k$ and $r_k$ and by \eqref{ukrep} (which implies $u(x)=0$ if $|x|>1$) we obtain that $\eta_k(x)=\log r_k -\alpha_k=\log 2- u_k(0)$ for $|x|>r_k^{-1}$. Then, we can assert that
$$\int_{|x|>r_k^{-1}}\frac{|\eta_k(x)|}{1+x^2}dx\leq C u_k(0)r_k\xrightarrow{k\to\infty}0.$$

On the other hand, again by definition of $\eta_k$ and $r_k$ and by \eqref{ukrep}, for $|x|<r_k^{-1}$ we have
\begin{align*}
\eta_k(x)-\log 2&=u_k(r_k x)-u_k(0)\\
&=\frac1\pi\int_{I}\log\frac{|y|}{|r_kx-y|}e^{\hat u_k(y)}dy+\frac1\pi\int_I(H(r_kx,y)-H(0,y))e^{\hat u_k(y)}dy \\
&=:f_k(x)+g_k(x).
\end{align*}
First, we bound the first integral as follows. Changing the variable $y\mapsto r_ky$ we obtain
$$f_k(x)=\int_{|y|<r_k^{-1}}\log\bra{\frac{|y|}{|x-y|}}e^{\eta_k(y)}dy,$$
and with Fubini's theorem we bound 
\begin{align}\label{I_R}
\int_{I_R}\frac{|f_k(x)|}{1+x^2}dx&\leq C\int_{|y|<r_k^{-1}}e^{\eta_k(y)}\int_{I_R}\left|\log\frac{|y|}{|x-y|}\right|\frac{dx}{1+x^2}dy,\quad I_R:=(-r_k^{-1},r_k^{-1})\setminus(-R,R).
\end{align}
We  claim that for $R$ sufficiently large $$\int_{I_R}\frac{|f_k(x)|}{1+x^2}dx<\ve.$$
By the previous bound \eqref{I_R}, this would follow immediately once we prove  \begin{align}\label{fk}\int_{I_R}\left|\log\frac{|y|}{|x-y|}\right|\frac{dx}{1+x^2}<\ve\quad\text{for }|y|\geq 1.\end{align} Note  that the inequality is trivial if $|y|<1$. Splitting $I_R$ into $$I_R=\cup_{i=1}^3A_i,\quad A_1:=\{|x|\leq \frac{|y|}{2}\}\cap I_R,\quad A_2:=\{|x|\geq 2|y|\}\cap I_R,\quad A_3:=I_R\setminus (A_1\cup A_2)$$ we write $$\int_{I_R}\left|\log\frac{|y|}{|x-y|}\right|\frac{dx}{1+x^2}=\sum_{i=1}^3\int_{A_i}\left|\log\frac{|y|}{|x-y|}\right|\frac{dx}{1+x^2}=:(I_1)+(I_2)+(I_3).$$ Using that $|x-y|\approx |y|$ on $A_1$ one gets $$(I_1)\leq C\int_{|x|>R}\frac{dx}{1+x^2}<\frac{\ve}{4}.$$ Since $|x-y|\approx |x|$ on $A_2$  $$(I_2)\leq C\int_{|x|\geq R}\frac{\log |x|}{1+x^2}dx<\frac{\ve}{4}\quad\text{for }|y|\geq 1.$$ 
Finally, we have $|y|\approx |x|$ on $A_3$, and using the assumption $|y|\geq 1$, we get for $R$ large enough
\begin{align*}(I_3)&\leq C\int_{|x|\geq R}\frac{\log |x|}{1+x^2}dx+C\min\{\frac{1}{R^2},\frac{1}{y^2} \}\int_{A_3}|\log|x-y||dx\\
&\leq \frac{\ve}{8} +C\min\{\frac{1}{R^2},\frac{1}{y^2} \}|y|\log(1+|y|)\\
&<\frac{\ve}{4}. 
\end{align*}
This proves \eqref{fk}. 

Using that $|H(x,y)|\leq C+|\log (1-|x|)|$, one easily gets $$\int_{I_R}\frac{|g_k(x)|}{1+x^2}dx<\ve\quad\text{for } R\gg1.$$  Step 2 follows. 

\medskip 

\noindent\textbf{Step 3} (Equicontinuity) For every $\ve>0$ and $R>0$ there exists $\delta=\delta(\ve,R)>0$ such that $$|\eta_k(x_1)-\eta_k(x_2)|<\ve\quad\text{for }|x_1-x_2|<\delta\text{ with }x_1,x_2\in (-R,R).$$

We have
\begin{align*}
\eta_k(x_1)-\eta_k(x_2)&=\frac1\pi\int_{I}\log\frac{|r_kx_2-y|}{|r_kx_1-y|}e^{\hat u_k(y)}dy+\frac1\pi\int_I(H(r_kx_1,y)-H(r_kx_2,y))e^{\hat u_k(y)}dy \\
&=:f_k(x_1,x_2)+g_k(x_1,x_2).
\end{align*}
It is easy to see, using the continuity of $H$, that $|g_k(x_1,x_2)|<\ve$ if $\delta>0$ is sufficiently small. For $M\gg R$ 
\begin{align*}
|f_k(x_1,x_2)|&\leq\frac1\pi\left( \int_{|y|\leq Mr_k}+\int_{Mr_k\leq |y|\leq 1}\right)\left|\log\frac{|r_kx_2-y|}{|r_kx_1-y|}\right|e^{\hat u_k(y)}dy\\
&=\frac1\pi\int_{|y|\leq M}\left|\log\frac{|x_2-y|}{|x_1-y|}\right|e^{\eta_k(y)}dy+\frac1\pi\int_{M\leq |y|\leq r_k^{-1}}\left|\log\frac{|x_2-y|}{|x_1-y|}\right|e^{\eta_k(y)}dy\\
&=:(I)+(II).
\end{align*}
As $\eta_k\leq \log 2$, for every fixed $M>0$ we can choose $\delta>0$ so that $(I)<\ve$.  Since $$\frac{|x_2-y|}{|x_1-y|}=1+|x_1-x_2|O(\frac1M)\quad \text{for }x_1,x_2\in (-R,R)\text{ and }|y|\geq M>>R,$$ one gets $$(II)\leq \frac CM|x_1-x_2|\int_{|y|\leq r_k^{-1}}e^{\eta_k(y)}dy \leq \frac CM|x_1-x_2|\leq C\frac\delta M.$$  

This proves Step 3. 

\medskip 

\noindent\textbf{Step 4} (Up to a subsequence) $\eta_k\to \eta$ in $C^0_{loc}(\R)$ where $\eta$ satisfies ($\mathcal{S}(\R)$ is the Schwartz space) $$\int_\R \eta(-\Delta)^\frac12\vp dx=\int_\R e^\eta\vp dx\quad\text{for every }\vp\in\mathcal{S}(\R).$$
This follows directly, by Ascoli-Arzel\'a Theorem, from Steps 2 and 3.
Then, by a classification results, see e.g. \cite[Theorem 1.8]{DMR} or \cite[Theorem 1.7]{DM}, $\eta= \eta_0$, and $ii)$ is proven. 

Moreover, as a corollary of $ii)$ we obtain $iii)$:
$$\lim_{R\to\infty}\lim_{k\to\infty}\int_{-Rr_k}^{Rr_k}e^{\hat u_k}dx=\lim_{R\to\infty}\int_{-R}^{R}\frac{2}{1+x^2}dx=2\pi.$$

\medskip 

\noindent\textbf{Step 5} We prove here $\alpha_k\to\infty$. 

Assume by contradiction that $\alpha_k\not\to\infty$.  Then for every $\ve>0$ and for $k$ large,   from \eqref{ukhatrep}, and together with $ii)$ 
$$\hat u_k(x)\geq \frac{1}{\pi}\int_{I}\log\left(\frac{1}{|x-y|}\right)e^{\hat u_k(y)}dy-C\geq \frac32\log \frac{1}{|x|}-C,\quad \ve\leq  |x|\leq 1 .$$ This 
 contradicts $\rho_k<2\pi$, thanks to Proposition \ref{poho}.  Thus, part $iv)$ is proved.

\medskip 

\noindent\textbf{Step 6}  $\hat u_k\to-\infty$ in $C^0_{loc}(\bar I\setminus\{0\})$.

Since $\hat u_k$ is monotone decreasing, it is sufficient to show that $\hat u_k(\ve)\to-\infty$ for every $\ve>0$.  As $\hat u_k(\frac \ve2)\not\to\infty$, which follows from $\rho_k<2\pi$ and the monotonicity of $\hat u_k$,   we have\begin{align*}\hat u_k(\ve)+\alpha_k&\leq  C(1+|\log\ve|)+\frac1\pi\int_{|y-\ve|<\frac\ve2}  \log\left(\frac{1}{|\ve-y|}\right)e^{\hat u_k(y)}dy\\
&\leq C(1+|\log\ve|)+Ce^{\hat u_k(\frac\ve2)} \\
&\leq C(\ve).\end{align*} 
This  bound, together with Step 5, implies Step 6. In this way, we have proved part $v)$, and with it, the whole Lemma. 
\end{proof}

\begin{lemma}\label{convGreen} For $\sigma\in (0,\frac12)$ we have
\begin{equation}\label{limuk}
u_k\to 2\pi G_0\text{ in }C^{0,\sigma}_{\loc}(\bar I\setminus \{0\}).
\end{equation}
\end{lemma}
\begin{proof} \textbf{$C^0$ convergence:} 
We write \begin{align*}u_k(x)-2\pi G_0(x)&=\int_I \left(G_x(y)-G_0(x)\right)e^{\hat u_k(y)}dy+(\rho_k-2\pi)G_0(x)\\
&=:v_k(x)+(\rho_k-2\pi)G_0(x).
\end{align*}
It follows that  $(\rho_k-2\pi)G_0\to0$ in $C^\infty(\bar I)$, thanks to Proposition \ref{poho} and part $iii)$ of Lemma \ref{conveta}. For $0<\ve\leq |x|\leq 1$  we have $$|G_x(y)-G_0(x)|\ll\ve\quad\text{if }|y|\ll\ve $$ and $$|G_x(y)-G_0(x)|\leq C+C|\log|x-y||+C|\log (1-|y|)|\quad\text{for } |y|<1.$$  This bound together with part $v)$ of Lemma \ref{conveta} would imply $v_k\to0$ in $C^0_{loc}(\bar I\setminus\{0\})$. 

We claim that  $$[u_k]_{C^{0,\frac12}((\ve,1))}\leq C(\ve)\quad\text{for every }  \ve>0. $$
Then the $C^{0,\sigma}_{\loc}(\bar I\setminus\{0\})$ convergence for $\sigma<\frac12$ will follow immediately from the Ascoli-Arzerl\`a Theorem.

For $x\in (\ve,1)$ and $h>0$  with $x+h\leq 1$ we have \begin{align*}u_k(x+h)-u_k(x)&=\frac1\pi\int_I\log\frac{|x-y|}{|x+h-y|}e^{\hat u_k(y)}dy+\int_I(H(x+h,y)-H(x,y))e^{\hat u_k(y)}dy\\
&=:I_1+I_2\end{align*} Since $$\log\frac{|x-y|}{|x+h-y|}=O_\ve(h)\quad \text{for }|y|\leq\frac\ve2,\quad x\geq\ve,\quad h>0,$$ and $\hat u_k\to-\infty$ in $C^0_{loc}(\bar I\setminus\{0\})$ by part $v)$ in Lemma \ref{conveta}, we obtain  
\begin{align*}
|I_1|\leq C_\ve h+C\int_{I}\left| \log\frac{|x-y|}{|x+h-y|}\right|dy\leq C_\ve h|\log h|.
\end{align*}
In order to bound $I_2$ we use  \begin{align}&H(x+h,y)-H(x,y)\notag\\&=\int_0^1\frac{\partial }{\partial t}H(x+th,y)dt \notag\\
&=\frac{h}{\pi}\int_0^1\frac{-y-(x+th)\sqrt{\frac{1-y^2}{1-(x+th)^2}}}{\sqrt{(1-(x+th)^2)(1-y^2)}+1-(x+th)y}dt \notag\\
&=O(h) \frac{1}{\sqrt{1-y^2}}\int_0^1\frac{dt}{\sqrt{1-(x+th)^2}}+O(h)\int_0^1\frac{\sqrt{1-y^2}}{1-(x+th)y}\frac{dt}{\sqrt{1-(x+th)^2}}\notag\\
&=O(h)\frac{1}{\sqrt{1-y^2}}\int_0^1\frac{dt}{\sqrt{1-(x+th)^2}} \notag\\
&=O(h)\frac{1}{\sqrt{1-y^2}}\int_0^1\frac{dt}{\sqrt{1-(x+th)}} \notag\\
&=O(1)\frac{1}{\sqrt{1-y^2}}(\sqrt{1-x}-\sqrt{1-x-h}) \notag\\
&=O(\sqrt h)\frac{1}{\sqrt{1-y^2}}, \label{est-H}
\end{align}
where, since $x+th\leq 1\quad \forall t\in(0,1)$, 2nd to 3rd equality follows from 
\begin{align*}\frac{1}{\sqrt{(1-(x+th)^2)(1-y^2)}+1-(x+th)y}& \leq \text{min}\left\{\frac{1}{\sqrt{(1-(x+th)^2)(1-y^2)}}, \frac{1}{1-(x+th)y} \right\} %\\ & \leq \text{min}\left\{\frac{1}{\sqrt{(1-(x+th)^2)}}, \frac{1}{1-(x+th)y {\color{red}1-|y|}} \right\} 
\end{align*}
and 3rd to 4th equality follows from $$\frac{\sqrt{1-y^2}}{1-(x+th)y}\leq \frac{\sqrt{1-y^2}}{1-|y|}\leq C\frac{1}{\sqrt{1-y^2}}.$$ Therefore
$$|I_2|\leq C\sqrt h\int_I\frac{1}{\sqrt{1-|y|}}e^{\hat u_k(y)}dy\leq C\sqrt h.$$
This proves our claim. 
\end{proof}

\section{Proof of Theorem \ref{trm2}}
We set $$X:= C^0([-1,1]),\quad \|u\|_X:=\max_{[-1,1]} |u(x)|.$$ We define $T_\rho:X\to X$ given by $$T_\rho(u)(x):=\rho\int_I G(x,y)\frac{e^{u(y)}}{\int_I e^{u(\xi)} d\xi}dy.$$

\begin{lemma}
For every $\rho>0$ the operator $T_\rho$ is compact. 
\end{lemma}
\begin{proof}
Let $(u_k)$ be a sequence of functions in $X$ such that $\|u_k\|_X\leq M$. Then, up to a subsequence, $$\int_{I}e^{u_k}dx\to c_0,$$ for some $c_0>0$.  Moreover,  there exists $C=C(M,\rho)>0$ such that  for every $x_1,x_2\in I$ \begin{align*}|T_\rho(u_k)(x_1)-T_\rho(u_k)(x_2)|&\leq C\int_I\left|\log\frac{|x_1-y|}{|x_2-y|}\right|dy+C\int_I|H(x_1,y)-H(x_2,y)|dy\\ &\leq C|x_1-x_2|^\frac12, \end{align*} where we have used that  $$|H(x_1,y)-H(x_2,y)|\leq C\sqrt{|x_1-x_2|}(1-|y|)^\frac{-1}{2} ,$$ which follows from \eqref{est-H}. Thus, the sequence $(T_\rho(u_k))$ is  bounded in $C^\frac12(I)$, and hence, it is pre-compact in $X$.  
\end{proof}

\medskip 

\noindent\emph{Proof of Theorem \ref{trm2} (completed).} 
Non-existence of solutions to \eqref{eq0}-\eqref{dir} for $\rho\ge 2\pi$ follows at once from Proposition \ref{poho}.

We will use the Schauder fixed-point theorem to prove that  $T_\rho$ has a fixed point (say) $u_\rho$ for every $\rho\in(0,2\pi)$, which by Lemma \ref{green} will be a solution to \eqref{eq0}-\eqref{dir}. Fix $\rho\in (0,2\pi)$, and consider any sequence $(t_k,u_k)\in (0,1]\times X$ such that $u_k=t_kT_\rho(u_k)$. Then $u_k$ satisfies \eqref{eq0}-\eqref{dir} with $\rho$ replaced by $\rho t_k <2\pi$. Therefore, by Theorem \ref{trm1} there exists $C>0$ such that $\|u_k\|_X\leq C$. Hence, by Schauder's theorem, $T_\rho$ has a fixed point in $X$, which is a solution to \eqref{eq0}-\eqref{dir}. 

For $\rho\in (0,2\pi)$ let $u_\rho\in X$ be a  fixed point of $T_\rho$. Since $T_{2\pi}$ does not have a fixed point, thanks to Proposition \ref{poho}, and, since $u_\rho(0)=\max_I u(\rho)$ by Lemma \ref{MP}, we must have $$u_\rho(0)\to\infty\quad \text{as }\rho\uparrow 2\pi.$$
\hfill$\square$

\medskip 

\section{Appendix}\label{Hopf}

We present here a self-contained proof of the non-local moving-plane technique in the simple case of an interval.
It will be based on the following non-local Hopf-type lemma, which is now a rather classical result (see e.g. \cite[Theorem 1]{CL}, \cite[Lemma 1.2]{GS} or \cite[Lemma 2.7]{IMS}). We present a proof here, since we could not find a reference fitting our assumptions, and the same result can be used in other fractional problems on an interval, see e.g. \cite{MM2}. 

\begin{lemma} [Hopf-type lemma]\label{HL} Let $w\in L^\infty(\R)\cap C^0(\R)$  be a solution to \begin{align*}\left\{ \begin{array}{ll}(-\D)^\frac12w(x)=c(x)w(x)&\quad\text{on }(a,0) \\ w(x)=-w(-x)&\quad\text{on }\R\\ w\leq 0&\quad\text{on }(-\infty,0)\end{array}\right.\end{align*} for some  bounded function $c$, and $a\in [-\infty,0)$. Assume that $w$ is $C^3$ in a neighborhood of the origin. Then $w\equiv 0$ on $\R$ if and only if $w'(0)=0$.  \end{lemma} 
\begin{proof} We assume by contradiction that $w\not\equiv 0$ on $\R$ and $w'(0)=0$. Then, as $w$ is an odd  function, we have $w(0)=w'(0)=w''(0)=0$. Hence, by Taylor expansion,  for some  $\delta>0$
\begin{align}\label{22}\begin{array}{ll}w(y)-w(x)=(y-x)w'(x)+\frac{(y-x)^2}{2}w''(x)+O((x-y)^3)\\w(x)=O(x^3),\quad w'(x)=O(x^2),\quad w''(x)=O(x), \end{array}  \end{align}
for every $x,y\in(-\delta,\delta) $.
For  $x<0$ near the origin we write
\begin{align*}
(-\D)^\frac12w(x)&=\frac1\pi  P.V.\int_{\R} \frac{w(x)-w(y)}{(x-y)^2}dy\\&=\frac1\pi \left(P.V.\int_{y<0}K(x,y)(w(x)-w(y))dy+2\int_{y<0}\frac{w(x)}{(x+y)^2}dy\right)\\ &=:\frac1\pi \left[P.V.(I)+(II)\right],\end{align*} where $$K(x,y):=\left(\frac{1}{(x-y)^2}-\frac{1}{(x+y)^2} \right)> 0\quad\text{on }(-\infty,0)\times(-\infty,0).$$
This implies that $w<0$ on $(a,0)$, since if $w(x)=0$ for some $x\in (a,0)$, we would have
$$0=(-\Delta)^\frac{1}{2}w(x)= -P.V.\int_{y<0}K(x,y)w(y)dy<0,$$
contradiction. Consequently, 
$w\leq -M$ on $(a_1,a_2)$ for some $M>0$ and $a<a_1<a_2<0$. For $x<0$ very close to the origin and for $|a_2|>>\ve>>|x|$  we split $(-\infty,0)$ into $(-\infty,0)=\cup_{i=1}^5A_i$ where  $$ A_1:=(2x,0),\quad A_2=(-\ve,2x), \quad A_3:=(a_1,a_2), \quad A_4:= (a_2,-\ve),\quad A_5:=(-\infty,a_1).$$ We now write $$(I)=\sum_{i=1}^5I_i,\quad I_i:=\int_{A_i}K(x,y)(w(x)-w(y))dy.$$ Using \eqref{22} we obtain $$\int_{A_1}\frac{w(x)-w(y)}{(x+y)^2}dy=O(x^2).$$ Therefore, as $I_1$ is in the PV sense, again by \eqref{22} $$I_1=O(x^2)+PV \int_{2x}^0\frac{w'(x)+\frac12(x-y) w''(x)+O((x-y)^2)}{x-y}dy=O(x^2).$$
From $$K(x,y)|x-y|^3\leq 4|x|\quad\text{and }K(x,y)\leq \frac{1}{(x-y)^2}\quad\text{for }y\in A_2,$$ and \eqref{22} one gets $$I_2=O(\ve)|x|.$$
Since $K(x,y)\approx |x|$ and $w(x)-w(y)\geq \frac M2$ for $y\in A_3$,  we obtain $$I_3\geq c_1 |x|\quad\text{for some }c_1>0.$$ Now we fix $\ve>0$ small enough so that $$|I_1|+|I_2|\leq \frac 14c_1|x|.$$ Then, for  $\ve>>-x>0$ we have $w(x)-w(y)>0$ for $y\in A_4$, which leads to $I_4>0.$ Recalling that $w\leq 0$ on $(-\infty,0)$, we have $w(x)-w(y)\geq w(x)=O(x^3)$ for $y\in A_5$, which  gives $I_5\geq O(x^4)$.  Thus $$(I)\geq \frac34 c_1|x|+O(x^4). $$ Note that $$(II)=O(x^2).$$ Combining these estimates we obtain $$0=(-\D)^\frac12w(x)+c(x)w(x)\geq \frac34 c_1|x|+O(x^2)+c(x)O(x^3)=\frac34 c_1 |x|+O(x^2)>0,$$ for $x<0$ sufficiently small, a contradiction.  
 \end{proof} 

\begin{trm}\label{MP_Ap} Let $u\in C^\frac{1}{2}(\R)\cap C^\infty(I)$ be a solution to
\[
\left\{
\begin{array}{ll}
(-\Delta)^\frac12 u=f(u)&\text{in }I\\
u=0& \text{in }\R\setminus I\\
u>0&\text{in }I,
\end{array}
\right.
\]
where $f$ is Lipschitz continuous, non-negative and non-decreasing. Then $u$ is even and $u(x)\ge u(y)$ for $0\le x\le y$.
\end{trm}

\begin{proof}
First, we claim that $u$ is monotone decreasing on $(1-\ve,1)$ for some $\ve>0$.  Although this follows from Lemma 1.2 in \cite{GS}, we shall give a simple self-contained proof.
We write
$$u(x)=\frac{1}{\pi}v(x)+ w(x),$$
where
$$v(x):=  \int_I\log\bra{\frac{1}{|x-y|}}f(u(y))dy,\quad  w(x):=\int_IH(x,y)f(u(y))dy,$$
where $H(x,y)$ is as in \eqref{Green}.
Differentiating under the integral sign one obtains $w'\leq C$ on $(0,1)$. For $h$ small we have
\begin{align*}v(x+h)-v(x)&=f(u(x))\int_I\log\bra{\frac{|x-y|}{|x+h-y|}}dy\\
&\quad +\int_I\log\bra{\frac{|x-y|}{|x+h-y|}}\left(f(u(y))-f(u(x))\right)dy\\
&=:v_1(x,h)+v_2(x,h).\end{align*}
Using that $u\in C^\frac12(\R)$ one gets
$$\lim_{h\to0}\frac{v_2(x,h)}{h}=O(1)\quad\text{on }I.$$ Computing the integral explicitly we obtain
$$\lim_{h\to0}\frac{v_1(x,h)}{h}=f(u(x))\left( \log(1-x)-\log(1+x)\right)\quad\text{on }I.$$
Thus, for $\ve>0$ sufficiently small
$$u'(x)\leq C +\frac{1}{\pi}f(u(x))(\log(1-x)-\log(1+x))<0\quad\text{on }(1-\ve,1),$$ proving the claimed monotonicity.    In particular, as $u=0$ on $I^c$ and $u>0$ on $I$, for $\lambda>1-\frac\ve2$ we have $$u_\lambda(x):=u(x_\lambda)-u(x)\leq0\quad\text{on }\Sigma_\lambda:=(-\infty,\lambda),\quad x_\lambda:=2\lambda-x.$$ We set  $$\lambda^*:=\inf\{\bar\lambda>0:u_\lambda\leq 0\text{ on }\Sigma_\lambda  \text{ for every }\lambda\geq\bar\lambda\}.$$ We claim now that $\lambda^*=0$. Otherwise there would be a sequence $\lambda_n\uparrow \lambda^*>0$ and $x_n\in \Sigma_{\lambda_n}$ such that 
$$\max_{\Sigma_{\lambda_n}}u_{\lambda_n}= u_{\lambda_n}(x_n)>0.$$
Moreover, since $u(x)= 0$ for $x\ge 1$ and $u>0$ in $I$, we must have $x_n\in (-1+2\lambda_n,\lambda_n)$.  Then, up to a subsequence, $x_n\to x_0\in [-1+2\lambda^*,\lambda^*]$ and $u_{\lambda^*}(x_0)=0$.
Now, on the one hand, using the equation we have
$$(-\Delta)^\frac12 u_{\lambda^*}(x) = f(u(x_\lambda))- f(u(x)) \le 0\quad \text{for }x\in (-1+2\lambda^*,\lambda^*).$$
%and in particular we can write
%$$(-\Delta)^\frac12 u_{\lambda^*}(x) = c(x)u_{\lambda*}(x)\quad \text{for }x\in \Sigma_{\lambda^*}$$
%with $c(x)\le 0$.
On the other hand, with the singular kernel definition for the fractional Laplacian, since  $u_\lambda^*\leq 0\text{ on }\Sigma_{\lambda^*}$, $u_{\lambda^*}(x_0)=0$ and $u_{\lambda^*}(x)=-u_{\lambda^*}(x_{\lambda^*})$, we can compute its value at $x_0\in [-1+2\lambda^*,\lambda^*]$:
\begin{equation*}
\begin{split}
(-\Delta)^\frac12 u_{\lambda^*}(x_0)&=\frac{1}{\pi}P.V. \int_{\R} \frac{u_{\lambda^*}(x_0)-u_{\lambda^*}(y)}{(x_0-y)^2}dy\\
&=\frac{1}{\pi}P.V. \int_{\Sigma_{\lambda^*}} u_{\lambda^*} (y)\left(\frac{1}{(x_0-2\lambda^*-y)^2}-\frac{1}{(x_0-y)^2}\right)dy\\
&\geq 0.
\end{split}
\end{equation*}
 Then, %applying the strong maximum principle to $u_{\lambda^*}$ on $\Sigma_{\lambda^*}$
  we conclude that $x_0=\lambda^*$. Hence, $$0=u'_{\lambda^*}(x_0)=-2u'(\lambda^*).$$
Moreover,  $u_{\lambda^*}<0$ in $(-1+2\lambda^*,\lambda^*)$, and this contradicts Lemma \ref{HL}. 
Thus $\lambda^*=0$ and $u_{0}\leq 0$. In a similar way one can show that $u_{0}\geq 0$. 
\end{proof}

\end{document}